%% file: sn-article.tex
\RequirePackage{amsthm}
\documentclass[sn-mathphys-num]{sn-jnl}

\usepackage{amsmath,amssymb,bbm,amsthm}
\usepackage{booktabs}
\usepackage{array} 
\usepackage{dsfont}
\usepackage{graphicx,framed}
\usepackage{url,bm}
\usepackage{float}
\usepackage{appendix}
\usepackage{algorithm}
\usepackage{algpseudocode} 
\usepackage{subcaption}
\usepackage{mathrsfs,multirow}
\usepackage[table]{xcolor} 

\usepackage{comment}
\usepackage{xspace}

\usepackage{xcolor}
\hypersetup{
  colorlinks,
  citecolor=blue,
  linkcolor=blue,
  urlcolor=blue}
  
\usepackage{cleveref}

\def\diag{\mathrm{diag}}

\newcommand{\tol}{{\tt Tol }}
\newcommand{\R}{\mathds{R}}
\renewcommand{\Re}{\R}

\newcommand{\col}[1]{\left\{#1\right\}}

\theoremstyle{thmstyleone}%
\newtheorem{theorem}{Theorem}
\newtheorem{proposition}[theorem]{Proposition}%

\theoremstyle{thmstyletwo}%

\theoremstyle{thmstylethree}%
\newtheorem{definition}{Definition}%

\begin{document}

\title[Scenario Tree Reduction via Wasserstein Barycenters]{Scenario Tree Reduction via Wasserstein Barycenters}

\author*[1,2]{\fnm{Daniel} \sur{Mimouni}}\email{daniel.mimouni@ifpen.fr}

\author[2]{\fnm{Paul} \sur{Malisani}}\email{paul.malisani@ifpen.fr}

\author[3]{\fnm{Jiamin} \sur{Zhu}}\email{jiamin.zhu@ifpen.fr}

\author[1]{\fnm{Welington} \sur{de Oliveira}}\email{welington.oliveira@minesparis.psl.eu}

\affil[1]{\orgdiv{Centre de Mathématiques Appliquées (CMA)}, \orgname{Mines Paris PSL}, \orgaddress{ \city{Sophia Antipolis}, \postcode{ 06560}, \country{France}}}

\affil[2]{\orgdiv{Department of Applied Mathematics}, \orgname{IFP Energies nouvelles}, \orgaddress{\city{Rueil-Malmaison}, \postcode{92063}, \country{France}}}

\affil[3]{\orgdiv{Department of Control Signals and Systems}, \orgname{IFP Energies nouvelles},\orgaddress{ \city{Rueil-Malmaison}, \postcode{92063}, \country{France}}}


\abstract{Scenario tree reduction techniques are essential for achieving a balance between an accurate representation of uncertainties and computational complexity when solving multistage stochastic programming problems. In the realm of available techniques, the Kovacevic and Pichler algorithm (Ann. Oper. Res., 2015 \cite{Kovacevic}) stands out for employing the nested distance, a metric for comparing multistage scenario trees. However, dealing with large-scale scenario trees can lead to a prohibitive computational burden due to the algorithm's requirement of solving several large-scale linear problems per iteration. This study concentrates on efficient approaches to solving such linear problems, recognizing that their solutions are Wasserstein barycenters of the tree nodes' probabilities on a given stage. We leverage advanced optimal transport techniques to compute Wasserstein barycenters and significantly improve the computational performance of the Kovacevic and Pichler algorithm. Our boosted variants of this algorithm are benchmarked on several multistage scenario trees. Our experiments show that compared to the original scenario tree reduction algorithm, our variants can be eight times faster for reducing scenario trees with 8 stages, 78\,125 scenarios, and 97\,656 nodes. 
}




\maketitle

\input{intro}

\input{reduction_tree}

\newpage
\section{Compliance with Ethical Standards}
Funding: None.\\
Conflict of Interest: None.\\
Ethical approval: This article does not contain any studies with human participants or animals performed by any of the authors.

\input{sn-article.bbl}

\end{document}

%% file: intro.tex
\section{Introduction} \label{introduction}

Stochastic optimization techniques have proved essential in solving optimization problems in the presence of uncertainties driven by variables such as fluctuating prices, unpredictable demand, supply variations, resource availability, and scheduling intricacies. The notion of stochastic programming is pioneered by Dantzig \cite{dantzig1955linear}, and applications can be found in various sectors, including the financial industry \cite{edirisinghe2005multiperiod, brodt1983min}, supply chains \cite{paulo2017designing}, management science \cite{carpentier2012dynamic}, energy economics \cite{analui2014distributionally,beltran2017application}, electrical markets \cite{dupavcova2003scenario}, hydro-thermal power systems \cite{de2010optimal} and maintenance of units \cite{qian2017scenario}.

Multistage stochastic programming, a class of stochastic optimization, often relies on scenario trees to represent the underlying stochastic process. In the quest for an accurate representation of uncertainties, large scenario trees are required.  However, as the number of scenarios increases, so does the complexity of the problem and the computational effort regardless of the optimization method, whether it is based on scenario decomposition or stage decomposition \cite{Beltran_2021}. Hence, finding a balance between an accurate representation of uncertainties and numerical tractability is of paramount importance in real-life applications modeled as multistage stochastic optimization problems.


In 2003, the influential work \cite{dupavcova2003scenario} introduced scenario reduction techniques based on the minimization of the Wasserstein distance between two scenario trees, pioneering the forward reduction and backward selection methodologies. 
Based on these ideas,  \cite{de2010optimal} proposes to compute the Wasserstein distance at different nodes of a multistage scenario tree to design a scenario tree reduction algorithm. Differently, \cite{li2014optimal} formulates the  scenario reduction problem as a mixed-integer linear programming  problem. To decrease the computational effort when dealing with the Wasserstein distance, which can be formulated as a linear programming (LP) problem for finite sample of scenarios, subsequent works \cite{li2016linear, kammammettu2023scenario} employ entropy-regularization schemes leveraged by the Sinkhorn–Knopp algorithm \cite{sinkhorn1967concerning}. 
However, being based on the Wasserstein distance, these techniques ignore the filtration (structure) of multistage scenario trees. 
Neglecting the tree filtration potentially leads to deviations in solutions, because the solution of a multistage stochastic optimization problem is non-anticipative. As an attempt to cope with this shortcoming, the work \cite{heitsch2006stability} takes into consideration the so-called filtration distance, which relies on the tree structure. The stability results in that paper are the basis for the scenario tree reduction approach proposed in \cite{heitsch2009scenario}, which is a Wasserstein-based method that measures distances between nodes with the same parent, ultimately reducing pairs of nodes until a stopping criterion is met. 
However, this method encounters computational challenges, as it relies on solving NP-hard facility location problems. The work \cite{chen2018scenario} proposes a scenario tree reduction algorithm that circumvents this difficulty by clustering tree nodes based on a new filtration distance and computes an approximation of the reduction problem. Scenario reduction strategies based on clustering are numerous in the literature, but often lack stability properties; see, for instance,  \cite{xu2012scenario,horejvsova2020evaluation} and references therein.

A milestone in the field of scenario tree reduction was the introduction of the \emph{nested distance} in \cite{pflug2010version}, subsequently studied in \cite{pflug2012distance}.
The nested distance, also called \emph{process distance}, offers a valuable framework for comparing multistage scenario trees as it considers the underlying filtrations. Leveraging the nested distance to guide scenario tree reduction enables control over the reduced tree's statistical quality and its impact on the objective value of the underlying multistage stochastic optimization problem. While it is relatively easy to calculate the nested distance between two given trees \cite{pichler2022nested}, finding a tree (with given filtration) that minimizes the nested distance is a much more challenging task. The reason is that the approximating tree's probabilities and support (i.e., the scenarios or outcomes) must be chosen to minimize the nested distance. This leads to a challenging optimization problem, which is large-scale and nonconvex. To tackle such a problem, Kovacevic and Pichler \cite{Kovacevic} introduced an algorithm that directly targets the minimization of the nested distance by alternating between probability and support optimizations. While the last task has an explicit solution (should the Euclidean norm be employed as a metric in the nested distance), optimizing the probabilities amounts to solving several large-scale LPs per stage. This represents the main bottleneck of the scenario tree reduction algorithm of \cite{pichler2022nested}, which can be partially avoided in some special cases. For instance, the work \cite{beltran2017application} provides a variant of the algorithm capable of efficiently handling large-scale multistage scenario trees provided the stochastic process is \emph{stage-wise independent}, a strong assumption we do not assume in this work. Hence, for general stochastic processes, there is a clear need for more practical scenario tree reduction approaches based on the nested distance. This work contributes in this direction by boosting the Kovacevic and Pichler (KP) algorithm \cite{pichler2022nested}. 

One of our contributions stems from recognizing that the algorithm's most time-consuming task, the probability optimization step, amounts to computing Wasserstein barycenters of the tree nodes' probabilities. In addition, we leverage advanced optimal transport techniques to compute Wasserstein barycenters within the KP's algorithm by employing efficient and dedicated approaches such as the Iterative Bregmann Projection method (IBP) \cite{IBP} and the   Method of Averaged Marginals (MAM) \cite{MAM}, to attractively boost the computational performance for reducing scenario trees. As a third contribution, we benchmark our variants of the KP algorithm on real-life data and empirically show that they significantly outperform the baseline KP algorithm for large-scale multistage scenario trees. Our findings alleviate the algorithm's bottleneck, helping thus to elevate the KP approach among the top algorithmic choices for scenario tree reduction. We also emphasize the importance of the filtration initialization in the finding of a reduced tree.

The remainder of this work is organized as follows. \Cref{Kovacevic and Picler's approach} presents the Kovacevic and Pichler's approach and recalls the mathematical background for the nested distance. The barycentric approaches are introduced in \Cref{Our new approach}, and the improved variant of the scenario tree reduction algorithm is developed. Some applications in \Cref{Applications} compare the new approach with the original algorithm by Kovacevic and Pichler, concluding with a speed-up method designed for real-life problems.

%% file: reduction_tree.tex

\section{The Kovacevic and Pichler's approach} \label{Kovacevic and Picler's approach}
In \cite{pflug2010version}, Pflug introduced the Nested Distance (ND), which is built on the Wasserstein distance, exploiting its structure and extending it to accommodate the specific characteristics of stochastic processes (see also \cite{pflug2012distance}). It is a tailor-made measure for stochastic processes: it inherits the foundational properties of the Wasserstein distance, such as its sensitivity to local structure and robustness to outliers, while providing a more nuanced and specialized measure for comparing the similarities between different realizations of stochastic processes.

\subsection{Scenario trees and notations}\label{tree_notations}
A T-period scenario tree ($\mathcal{N}$, A) is a discrete form of a random process - a family of random variables 
(see \cite{Kovacevic} for more). 
It is a composed of (a set of) nodes $\mathcal{N}$ and edges A.
\if{
For the scenario tree, there is only one root node associated with the value $\xi_0(\omega)$, for all $\omega\in \Omega$.
It is accepted to call the vertices $\mathcal{N}$, nodes \cite{Pflug_Pichler_2014}, A is the set of edges.  
    
Thus, we denote the value taken by the process $\xi \in \Xi$ at node $n$ by $\xi_t(\omega)=\xi(n)$. 
}\fi

A node $m\in \mathcal{N}$ is a direct predecessor or parent of the node $n\in\mathcal{N}$, if $(m, n)\in A$, where $(m, n)$ denotes the edge between the nodes $m$ and $n$.
This relation is embodied by the notation $m = n -$. Reciprocally, $n$ is a direct successor (or child) of $m$ and this set is denoted $m+$, such that $n\in m+$ if and only if $m = n -$. 
In the same vein, we denote the predecessors (or ancestors) of $n \in \mathcal{N}$ as $\mathcal{A}(n)$, the set of nodes for which there exists a path to $n$: for example $m_1 = m_2 -$ and $m_2 = n -$ (equivalently $n \in m_2+$ and $m_2 \in m_1+$), then $m_1, m_2 \in \mathcal{A}(n)$.
We consider only trees with a single root, denoted by 0, i.e., $0 - = \emptyset$, but the methods developed here can be generalized for multi-root trees (forests). 
Nodes $n_T \in \mathcal{N}$ without successor nodes (i.e., $n_T+ = \emptyset$) are called leaf nodes. For every leaf node $n_T$ there is a sequence $\xi_{n_T}:=(n_0, \ldots, n_t , \ldots, n_T)$ where $n_0 \in \mathcal{A}(n_1)$, $n_1 \in \mathcal{A}(n_2)$ etc, from the root to the leaf node $n_T$ composed by $T+1$ nodes.
$\mathcal{F}_t$ is the filtration generated by the sigma-algebra of the family $\left(\xi_n\right)_{n\in\mathcal{N}_t}$ and we define $\mathcal{F}:=\left(\mathcal{F}_t\right)_{t=1,\dots, T}$.
The nodes bear a value named \emph{quantizer}: $\xi:\mathcal{N}\mapsto \Xi$, we call $\xi(n)$ the quantizer of node $n$, where $\Xi$ is a finite dimension metric space.
We denote the probability, assigned to node $n$, by $P(n)$, that satisfies: $P(n)=\sum_{\tilde n \in n+}P(\tilde n)$ for $n\in\mathcal{N}$ and $\sum_{n\in \mathcal{N}_T}P(n) = 1$. We denote conditional probability between successors by $P(n|m) = P(n)/P(m)$ for $m = n -$. Furthermore, using a distance $\mathtt{d}:\Xi^{t+1}\times\Xi^{t+1}\rightarrow\R$, we denote the distance between two nodes $n_1,n_2$ at stage $t$, respectfully, as:
\begin{equation} \label{1}
\mathtt{d}_{n_1,n_2} :=  \displaystyle \sum_{n\in \xi_{n_1}, \bar{n}\in \xi_{n_2}, n \text{ and } \bar{n} \text{ at stage }t}\tilde{\mathtt{d}}(\xi(n),\xi(\bar{n})) 
\end{equation} 
Where $ \tilde{\mathtt{d}}$ is a distance on $\Xi$.

\begin{figure}
    \centering
    \includegraphics[width=0.7\textwidth]{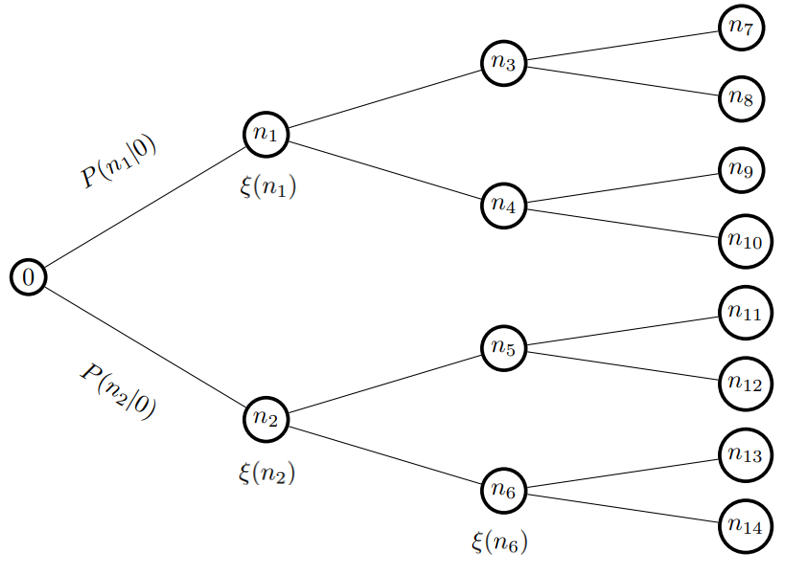}
    \caption{Scenario tree notations.}
    \label{fig:illustration_filtration}
\end{figure}

\if{
In what follows we only define trees but there is an equivalence between a familly of discret stochastic processes joined with its filtration and a scenario tree (see \cite{Kovacevic}). A scenario can be described as a random vector $\xi$ having a probability measure $\mu$, it can be compared to another scenario tree with a distance function defined between probability distributions such as the Wasserstein distance. In fact, the Wasserstein distance has been widely employed to compare trees and is the base of a lot of reduction tree methods. However it can be shown that it is not suitable to distinguish stochastic processes with different flows of information (see 2.1 in \cite{Kovacevic}). When comparing trees, we do not want to only compare the scenarios but also the scenarios and the filtration (flow of information) affiliated with them. While the Wasserstein distance is efficient for comparing probability distributions, it is not fitted to take into account the evolution of the available information along the stages of stochastic processes \cite{heitsch2006stability, Kovacevic}. On the other hand, the Nested Distance is able to take explicitly into account this flow of increasing information, being a generalization of the Wasserstein Distance.

}\fi

\subsection{Nested Distance for Trees}

Using the tree notations introduced in \Cref{tree_notations}, the process distance of order $\iota$ between two trees $\mathbf{P}:=$($\Xi^{T+1}, \mathcal{F}, P$) and $\mathbf{P}':=$($\Xi^{T+1}, \mathcal{F}', P'$) is the \emph{discrete Nested Distance for Trees} (NDT). The transport mass between node $i\in\mathcal{N}_t$ and node $j\in\mathcal{N}_t'$ at stage $t\in\{1,\dots, T\}$, is noted $\pi_{i,j}$ or $\pi(i,j)$.


\begin{definition}[Nested Distance for Trees] For $\iota \in [1, \infty)$, the process distance of order $\iota$ between $\mathbf{P}$ and $\mathbf{P'}$ is the $\iota^{th}$ root of the optimal value of the following LP:

\begin{equation}\label{NDT}
\tag{NDT}
{\rm ND}_\iota(\mathbf{P},\mathbf{P}'):=
\left\{
\begin{array}{lll}
\displaystyle \min_{\pi} & \displaystyle \sum_{i\in\mathcal{N}_T, j\in\mathcal{N}_T' } \pi(i,j) \mathtt{d}_{i,j}^\iota\\
\\
\mbox{s.t.}& \sum_{\{j:n\in\mathcal{A}(j)\}}\pi(i,j|m,n)=P(i|m),& (m\in\mathcal{A}(i),n)\\
\\
&\sum_{\{i:m\in\mathcal{A}(i)\}}\pi(i,j|m,n)=P'(j|n),& (n\in\mathcal{A}(j),m)\\
\\
&\pi_{i,j}\geq0 \text{ and } \sum_{i,j}\pi_{i,j}=1.

\end{array}
\right.
\end{equation}
\end{definition}
\if{
This linear programming problem (LP) can be rather large and challenging to solve directly, this is why Kovacevic and Pichler \cite{Kovacevic} introduced an efficient recursive method. Then \cite{pichler2022nested} developed an even faster algorithm based on the Sinkhorn algorithm. In the literature, this distance is used and conveniently computed with the recursive exact LP resolution or the Sinkhorn alternative algorithm.
}\fi
Note that \eqref{NDT} is a generalization of the Wasserstein distance. Indeed, the transport plan $\pi$ does not only respect the marginals imposed by $P$ and $P'$ but also respects the conditional marginals. These constraints embed the filtration in the definition of the distance between trees. Note that in the following, the term \emph{ND} is used to refer to the value of \eqref{NDT}. 
\\

\subsection{The Kovacevic and Pichler algorithm for scenario tree reduction} \label{KP algo}
The scenario tree reduction mechanism involves solving structure-like \eqref{NDT} problems, between the original tree and the smaller one. The latter has given filtration (same number of stages but considerably fewer number of nodes than the original tree), initialized probabilities and quantizer values. 

\if{\\
Even though, the distance \eqref{NDT} to measure the closeness between scenario trees is broadly adopted, no effective method to reduce general scenario trees based directly on this distance has been proposed since Kovacevic and Pichler \cite{Kovacevic} (2015), because this approach has been judged too computationally demanding. An exception is \cite{beltran2017application}, which focuses on the particular class of stage-wise independent scenario trees.
}\fi
The scenario reduction optimization problem is non-convex, due to the optimization of both quantizers and probabilities. The method in \cite{Kovacevic} operates a classical block coordinate optimization scheme. As illustrated in \Cref{algo_KP_dessin}, after a filtration is chosen, the first step is the optimization of the probabilities $P'$ for fixed quantizers, and the second step is the optimization of the quantizer values $\{\xi'(n)\in \Xi:n\in\mathcal{N}'\}$ for fixed probability. The latter step has an exact analytical solution in the Euclidean case, i.e. $\iota=2$, and  when $\tilde{\mathtt{d}}$ in \eqref{1} is the Euclidean norm \cite{Kovacevic}. The probability optimization is more difficult because it requires solving multiple (potentially large-scale) LPs. In the remainder of this work, we will only develop the probability optimization, the quantizer optimization will only be recalled briefly in \Cref{alg_reduc}.\\

\begin{figure}[h] 
    \centering
    \includegraphics[width=1\textwidth]{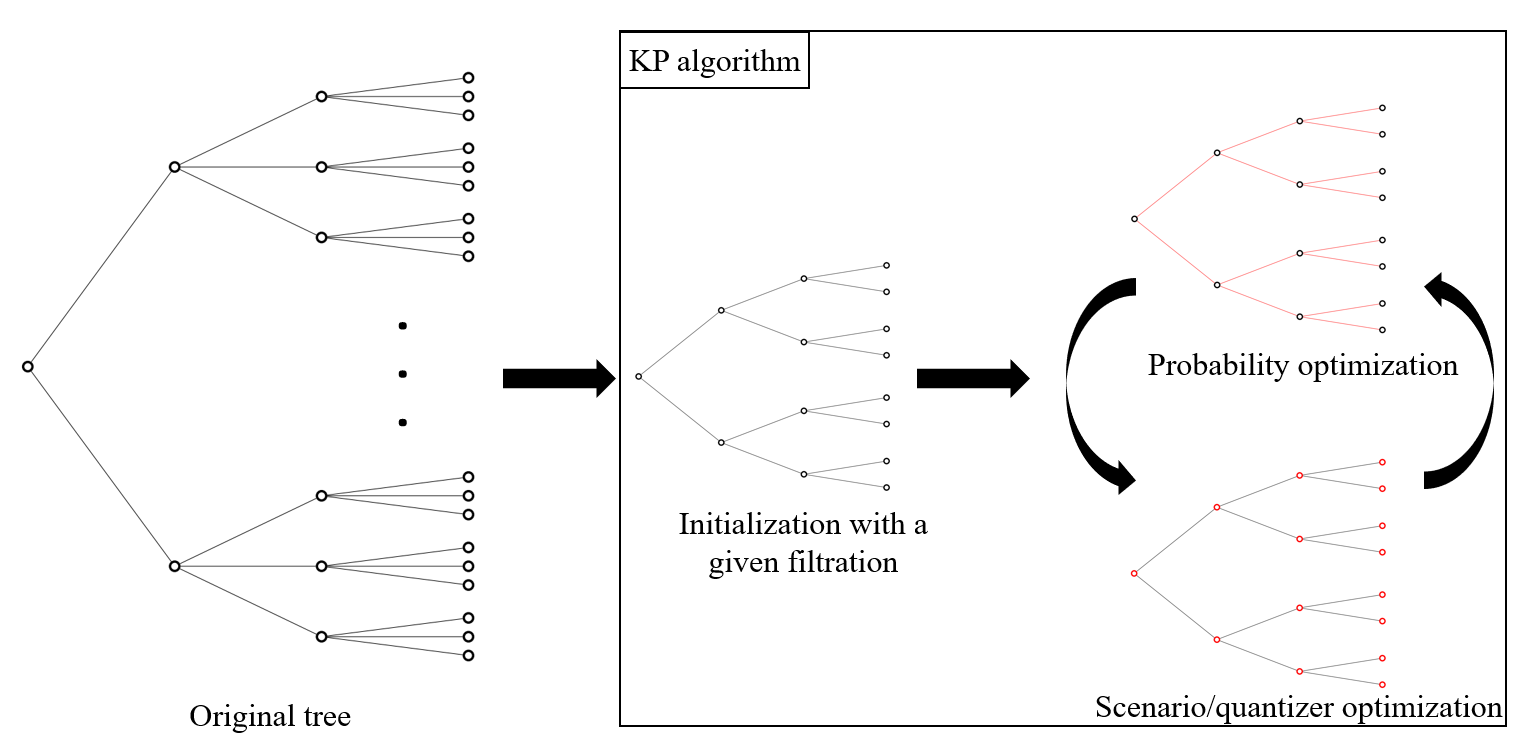}
    \caption{A general scheme of the Kovacevic and Pichler algorithm: to approximate a tree, a smaller tree with a given filtration is improved in order to minimize the nested distance with the original tree. The probabilities and the quantizers are alternatively optimized until convergence.} \label{algo_KP_dessin}
\end{figure}

The probability optimization steps can be stated as follows: given the stochastic process quantizers $\{\xi'(n)\in \Xi:n\in\mathcal{N}'\}$ and structure of $(\mathcal{N}', A')$, we are looking for the optimal probability measure $P'$ to approximate $\mathbf{P}:=$($\Xi^{T+1}, \mathcal{F}, P$), regarding the nested distance.  Inspecting formulation \eqref{NDT}, it turns out that $P'$ can be computed by jointly optimizing with respect to $\pi_{i,j}$  and $P'(j|j -)$.
This leads to the following large non-convex optimization problem:
\begin{equation}  \label{reduc_pb}
\left\{
\begin{array}{llllllllll}
\displaystyle \min_{\pi, P'} & \displaystyle \sum_{i\in\mathcal{N}_T, j\in\mathcal{N}_T' } \pi(i,j) \mathtt{d}_{i,j}^\iota\\
 \mbox{ }\\
\\
\mbox{s.t.}& \sum_{j\in n+}\pi(i,j|m,n)=P(i|m),& (\forall m\in\mathcal{A}(i),n)\\
\\
&\sum_{i\in m+}\pi(i,j|m,n)=P'(j|n),& (\forall n\in\mathcal{A}(j),m)\\
\\
&\pi_{i,j}\geq 0 \text{ and } \sum_{i,j}\pi_{i,j}=1\\
\\
&P'(j|j -)\geq 0.

\end{array}
\right.
\end{equation}
This is a bilinear problem, hence difficult to handle: there is a large number of decision variables and bilinear constraints (issued by the conditional probabilities in the second group of constraints, which involve the decision variables composing $\pi$ and $P'$).
Using the conditional probabilities  $\pi(i,j) = \pi(i,j|m,n) \times \pi(m,n)$, we can derive the recursive formula \eqref{subpb} below.\\
Let $\delta_\iota(m,n):= \sum_{i\in m+, j\in n+}\pi(i,j|m,n) \delta_\iota(i,j)$ for $m\in\mathcal{N}_t, n\in\mathcal{N}_t'$, and $\delta_\iota(i,j)= \mathtt{d}_\iota(\xi_i, \xi'_j)^\iota=: \mathtt{d}_{i,j}^\iota$ for the leaves $i$, $j$ of the trees.
\begin{subequations}
\begin{align}\label{subpb}
    \sum_{i\in\mathcal{N}_T, j\in\mathcal{N}_T' } \pi(i,j) \mathtt{d}_{i,j}^\iota &= \sum_{i\in\mathcal{N}_T, j\in\mathcal{N}_T' } \pi(i,j) \delta_\iota (i,j)\\
    &= \sum_{i\in\mathcal{N}_T, j\in\mathcal{N}_T' } \sum_{m \in i-, n \in j-} \pi(i,j\vert m, n) \pi(m, n) \delta_\iota(i,j)\\
    &=\sum_{n\in\mathcal{N}_{T-1}'}\sum_{m\in\mathcal{N}_{T-1}}\pi(m,n)\underbrace{\sum_{i\in m+, j\in n+}\pi(i,j|m,n) \delta_\iota(i,j)}_{\delta_\iota(m,n)} \\
     &= 
    \sum_{n\in\mathcal{N}_{T-1}'}\sum_{m\in\mathcal{N}_{T-1}}\pi(m,n) \delta_\iota(m,n) . 
\end{align}
\end{subequations}

Note also $\delta_\iota(0,0) = {\rm ND}(\mathbf{P}, \mathbf{P}')$, recursively. Thanks to this recursive formula, the problem can be split into recursive smaller problems for $m\in \mathcal{N}_t$ and $n\in\mathcal{N}_t'$: the conditional probability $\pi(\cdot,\cdot|m,n)$ is a solution to 
\begin{equation}  \label{recursive_pb}
\tag{RP}
\left\{
\begin{array}{llllllllll}
\displaystyle \min_{\pi} & \sum_{m\in\mathcal{N}_t} \pi(m,n)\sum_{i\in m+, j\in n+}\pi(i,j|m,n)\delta_\iota(i,j) \\
 \mbox{ }\\
\\
\mbox{s.t.}& \sum_{j\in n+}\pi(i,j|m,n)=P(i|m),& (i \in m+)\\
\\
&\sum_{i\in m+}\pi(i,j|m,n)= \sum_{i\in \tilde m +} \pi(i,j|\tilde m,n) ,& (j\in n+ \text{ and } m,\tilde m \in \mathcal{N}_t)\\
\\
&\pi(i,j|m,n)\geq 0.

\end{array}
\right.
\end{equation}
This reformulation of the problem is still bilinear, however, to overcome this difficulty, \cite{Kovacevic} proposes to fix $\pi(m,n)$ with the values computed from the previous iteration (or initialized at first iteration) giving rise to a LP approximation. The authors have empirically shown  that after few iterations of their algorithm the values assigned to $\pi(m,n)$ stabilize.

Despite this approximation, the problem is still challenging due to its huge dimensions. Thus,
the method's main limitation is its computational burden that becomes prohibitive for large-scale scenario trees, as exemplified in \Cref{Applications}. The reason is that the method requires solving potentially large-scale LPs as in \eqref{recursive_pb} repeatedly. We address the challenge by noticing that \eqref{recursive_pb} with fixed $\pi(m,n)$ for $m\in\mathcal{N}_t$ for a given $n$ is a \emph{Wassertein barycenter problem}  \eqref{WB}, for which specialized and efficient algorithms exist. 


\section{The Probability Optimization Step \eqref{recursive_pb} is a Wasserstein Barycenters Problem} \label{Our new approach}
We start with the Wasserstein distance definition.
\if{
\begin{definition}[Wasserstein Distance] \label{WD def}
Let $(\Omega,\mathtt{d})$ be a metric space and $P(\Omega)$ the set of Borel probability measures on $\Omega$. 
For $\iota \in [1, \infty)$, and probability measures $\mu$ and $\nu$ in  $P(\Omega)$. Their $\iota$-Wasserstein distance $W_\iota$ is :
\begin{equation}\label{discret_WD}
\tag{WD}
W_\iota(\mu,\nu):=\left(\min_{\pi \in U(\mu,\nu)} \iint_{\Omega\times \Omega} \mathtt{d}(\xi,\xi')^\iota d\pi(\xi,\xi')\right)^{1/\iota},
\end{equation}
where $U(\mu,\nu)$ is the set of all probability measures on $\Omega\times \Omega$ having marginals $\mu$ and $\nu$. 
We denote by $W_\iota^\iota(\mu,\nu)$, $W_\iota$ to the power $\iota$ , i.e.  $W_\iota^\iota(\mu,\nu) := (W_\iota(\mu,\nu))^\iota$. 
\end{definition}
With, for $\tau\geq 0$ a given scalar, $\Delta_n(\tau):=\col{u \in\R^n_+:\; \sum_{i=1}^nu_i=\tau}$ is the set of non negative vectors in $\R^n$ adding up to $\tau$.
If $\tau=1$, then $\Delta_n(\tau)$, denoted simply by $\Delta_n$, is the $n+1$ simplex.\\

}\fi

\begin{definition}[Wasserstein Distance] 
Given two probability measures $\mu,\nu \in P(\R^d)$, their Wasserstein distance is the $\iota$-th root of 
\begin{equation}\label{WD}
W_\iota^\iota(\mu,\nu):= \displaystyle \min_{\pi \in U(\mu,\nu)} \langle D, \pi \rangle.
    \tag{WD}
\end{equation}

\end{definition}
Where $\langle D, \pi \rangle:=\sum_{r,s}D_{r,s}\pi_{r,s}$. The cost matrix $D$ is composed by the distance values $(\mathtt{d}_{r,s})_{(r \times s)\in[1,R]\times[1,S]}$, where $R$ and $S$ are the support sizes of $\mu$ and $\nu$ respectively; $\pi$ is the transport matrix between the two probability densities; and $U(\mu,\nu)$ is the set of all transport plans having marginals $\mu$ and $\nu$.

\begin{definition}[Wasserstein Barycenter] 
Given $M$ measures
$\{\nu^1,\ldots,\nu^M\}$ in $ P(\R^d)$, an $\iota$-\emph{Wasserstein barycenter} with weights $\alpha \in \Delta_M$ is a solution to the following optimization problem:
\begin{equation}\label{WB2}
\min_{\mu \in P(\R^d)}\; \sum_{m=1}^M \alpha_m W_\iota^\iota(\mu,\nu^m)\,.
\end{equation}
\end{definition}

In what follows we present one of our contributions.
\begin{proposition}
The recursive problem \eqref{recursive_pb} with fixed $\pi(m,n)$ for $m\in\mathcal{N}_t$ and a given $n$ is a Wasserstein Barycenter problem.
\end{proposition}
\begin{proof}
Let $M$ empirical (discrete) measures $\nu^m$ having finite support sets:
\begin{equation}\label{eq:empirical}
{\tt supp}(\nu^m):=\col{\xi'^m_1,\ldots,\xi'^m_{S^m}}
\quad \mbox{and} \quad   \nu^m=\sum_{s=1}^{S^m} q^m_s \underline{\delta}_{\xi'^m_s},
\end{equation}
with $\underline{\delta}_u$ the Dirac unit mass on $u \in \Re^d$ and $q^m\in \Delta_{S^m}$, $m=1,\ldots,M$.\\
A barycenter $\mu:= \sum_{r=1}^R  p_r \underline{\delta}_{\xi_r}$ with $p \in \Delta_R$ of the family $(\nu^m)_{m\in[1,M]}$ is a solution to \eqref{WB2}.

Problem \eqref{WB2} can be equivalently formulated as a LP, where $\pi^m$ for $m=1,\dots,M$ denote the transport plan between the $\mu$ and the probability measures, see Section 2 of \cite{MAM}:

\begin{equation}  \label{HugeLP}
\left\{
\begin{array}{llllllllll}
\displaystyle \min_{p,\pi} & \displaystyle  \alpha_1 \sum_{r=1}^R \sum_{s=1}^{S^1} \mathtt{d}^1_{rs}\pi^1_{rs}&+\cdots +& \alpha_M \displaystyle\sum_{r=1}^R
 \sum_{s=1}^{S^M}  \mathtt{d}^M_{rs}\pi^M_{rs}\\
 \mbox{ }\\
 \mbox{s.t.} & \sum_{r=1}^R \pi^1_{rs} &&&\hspace{-1.3cm}= q^1_s,&\;s=1,\ldots,S^1 \\
 &&\ddots &&\hspace{-1.3cm}\vdots\\
 &&&\sum_{r=1}^R \pi^M_{rs} &\hspace{-1.3cm}= q^M_s,&\;s=1,\ldots,S^M\\[1em]
 &\sum_{s=1}^{S^1} \pi^1_{rs} &&&\hspace{-1.3cm}= p_r,&\;r=1,\ldots,R\\
 &&\ddots &&\hspace{-1.3cm}\vdots\\
  &&&\sum_{s=1}^{S^M} \pi^M_{rs} &\hspace{-1.3cm}= p_r,&\;r=1,\ldots,R\\
  \\
  & p \in \Delta_R,\pi^1\geq 0&\cdots &\pi^M\geq 0,
\end{array}
\right.
\end{equation}

Let us now show that~\eqref{recursive_pb} fits into this structure.\\ 
Let $\alpha_{m_i}^n:=\pi(m_i,n)$ for $i=1,\dots,M$, $M=|\mathcal{N}_t|$, and $(P'(j|n))_{j\in n+}$ an auxiliary vector of probabilities allowing to reformulate and couple the last group of constraints in~\eqref{recursive_pb}:$\sum_{i\in m+}\pi(i,j|m,n)=P'(j|n)$ for $j\in n+$, for all $m=m_1,\dots,m_M$. 
Given $t\in\{1,\dots,T\}$ and $n\in\mathcal{N}_t'$, problem \eqref{recursive_pb}  reads as:
\begin{equation}  \label{WB}
\tag{WB}
\left\{
\begin{array}{llllllllll}
\displaystyle \min_{P', \pi} & \displaystyle \alpha_1^n\sum_{i\in m_1+, j\in n+}\pi(i,j|m_1,n)\delta_\iota(i,j)&+\cdots +& \alpha_M^n \displaystyle \sum_{i\in m_M+, j\in n+}\pi(i,j|m_M,n)\delta_\iota(i,j)\\
 \mbox{ }\\
 \mbox{s.t.} & \sum_{j\in n+}\pi(i,j|m_1,n) &&=P(i|m_1),&\hspace{-3.4cm}\;(i \in m_1+) \\
 [1em]
 &&\hspace{-3cm}\ddots &&\hspace{-5.5cm}\vdots\\
  &&\hspace{-2.3cm} \sum_{j\in n+}\pi(i,j|m_M,n) &=P(i|m_M),&\hspace{-3.4cm}\;(i \in m_M+) \\[1em]
 &\sum_{i\in m_1+} \pi(i,j|m_1,n)  &&= P'(j|n),&\hspace{-3.4cm}\;(j\in n+)\\
 &&\hspace{-3cm}\ddots &&\hspace{-5.5cm}\vdots\\
 &&\hspace{-2.3cm}\sum_{i\in m_M+} \pi(i,j|m_M,n) &= P'(j|n),&\hspace{-3.4cm}\;(j\in n+)\\
  \\
  & \sum_{j\in n+}P'(j|n)=1,\pi(i,j|m_1,n)\geq 0&\cdots &\pi(i,j|m_M,n)\geq 0,
\end{array}
\right.
\end{equation}
which is equivalent to
\[
\displaystyle \min_{P'(.|n)\geq 0} \sum_{m\in\mathcal{N}_t} \alpha_m^n W_\iota^\iota((P'(j|n))_{j\in n+},(P(i|m))_{i\in m+}) \quad\mbox{s.t.}\quad \sum_{j\in n+}P'(j|n)=1,
\]
i.e., a Wasserstein Barycenter problem.

\end{proof}


To further illustrate our interpretation of problem~\eqref{recursive_pb} as a Wassertein Barycenter problem, see \Cref{bary_tree}.
The boxed subtree of the approximated (smaller) tree forms a probability measure with support $(n_7, n_8)$ and probability $(P(n_7|n_3), P(n_8|n_3))$; the probabilities of this subtree helping to minimize the ND is a barycenter of the (original tree's) subtrees with initial nodes $m_3, m_4, m_5, m_6$.
The Wasserstein Barycenter $(P(n_7|n_3), P(n_8|n_3))$ is computed by solving~\eqref{WB}, the next (approximated tree's) subtree, that is the one issued by node $n_4$ is solved in the same manner etc.
The probability $(P(n_9|n_4), P(n_{10}|n_4))$ are computed by solving another barycenter problem, but with different weights $ \alpha_m^{n_4}$ (and costs) for $m\in\{m_3, m_4, m_5, m_6\}$ (see \cref{WB}).
\begin{figure}[h] 
    \centering
    \includegraphics[width=1.\textwidth]{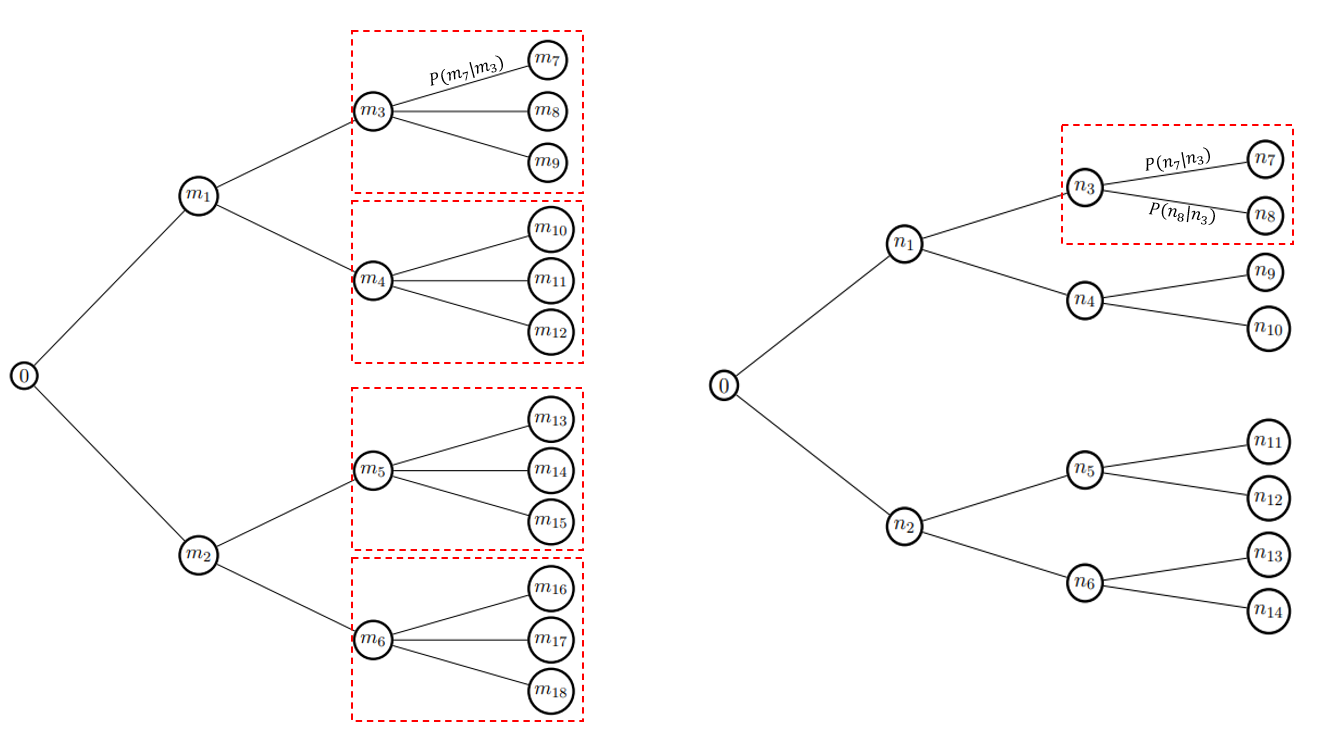}
    \caption{\textit{(left)} Original tree, \textit{(right)} Approximated tree. The probabilities $(P'(n_7|n_3), P'(n_8|n_3))$ are computed as the Wasserstein Barycenter of the set of (known) probabilities associated to the boxed subtrees on the left.}
    \label{bary_tree}
\end{figure}

It is thus clear that several Wasserstein Barycenter problems must be solved at every iteration of the scenario tree reduction algorithm \cite{Kovacevic}. The faster the computation of such barycenters, the faster the Kovacevic and Pichler's algorithm.

\subsection{Wasserstein barycenters techniques for scenario reduction}

Problem from \eqref{WB} is an LP and could, in principle, be solved by LP solvers such as Gurobi, Cplex, HiGHs, and others. However, in real life applications of scenario tree reduction, problem \eqref{WB} (equivalently problem from \eqref{recursive_pb})  can have huge dimensions and be  intractable by LP solvers, thus hindering the whole scenario tree reduction process. 
Many methods can be employed to tackle problem \eqref{WB} (for e.g. \cite{Cuturi_Doucet_14,J.Ye,IBP}). 
Note that these methods are typically employed to retrieve a barycentric probability measure. However, in the context of this work, we instead focus on directly extracting the transport plans between the measures and the barycenter. We will further develop these concepts in the following sections. In the remainder of the paper, we use two different WB algorithms, namely the Method of Averaged Marginals (MAM) \cite{MAM} and the Iterative Bregmann Projection algorithm (IBP) \cite{IBP}. 
In the following, we simplify the notation by denoting: $\pi(\cdot,\cdot|m_{\hat m},n):=\pi^{\hat m}\in \R^{R\times S^{\hat m}}$ for all $\hat m=1,\dots, M$, since $n$ is fixed in \eqref{WB}, and $P(\cdot|m)$ in \eqref{WB} is denoted by $q^{\hat m}$.

\subsubsection{The Method of Averaged Marginals (MAM)}\label{MAM_recall}

MAM's main idea \cite{MAM} is to reformulate \eqref{WB} as the problem of finding a zero of the sum of two maximal monotone operators, which can be solved by the celebrated Douglas-Rachford algorithm \cite{Bauschke_Combettes_2017,Douglas_Rachford_1956}.
When specialized to \eqref{WB}, the paper \cite{MAM} shows that Douglas-Rachford algorithm boils down to \Cref{alg_MAM}.

\begin{algorithm}
\caption{\sc \sc Method of Averaged Marginals - MAM}
  \label{alg_MAM}
\begin{algorithmic}

\State {\bf Input}: Initial plan
 $\pi=(\pi^{1},\ldots, \pi^{m}) $ and parameter  $\rho>0$ 
  \State Set $S^{m} \gets \mid$ {\tt supp}$(q^m) \mid$, for $m=1,\dots,M$
 \State Define  $a_m \gets
 (\frac{1}{S^{m}})/(\sum_{j=1}^M\frac{1}{S^{j}})$ and set $p^{m} \gets \sum_{s=1}^{S^{m}}\pi_{rs}^{m}$, $m=1,\ldots,M$
 \State Set $D^{m} \gets \alpha_m \left( \delta_\iota (i,j) \right)_{(i,j)\in m+ \times n+}$ and set $q^m=\left(P(i|m)\right)_{i\in m+}$

\Statex
\While{not converged}
\Statex
\State $p\gets \sum_{m=1}^M a_mp^{m}$ \Comment{Average the marginals}
\Statex
\For{$m=1,\ldots,M$}
\For{$s=1,\ldots,S^{m}$} 
\State $\pi^{m}_{:s} \gets {\tt Proj}_{\Delta(q_s^{m})} \Big(\pi^{m}_{:s} +2\frac{p-p^{m}}{S^{m}}- \frac{1}{\rho}D^{m}_{:s} \Big)-\frac{p-p^{m}}{S^{m}}$
\EndFor
\State $p^{m} \gets \sum_{s=1}^{S^{m}}\pi_{rs}^{m}$ \Comment{Update the $m^{th}$ marginal}
\EndFor
\Statex
\EndWhile
  \end{algorithmic}
\end{algorithm}
The projection step can be performed  \emph{exactly}, and in parallel, by using efficient specialized methods \cite{Condat_2016}. We employ the notation  $\pi^{m}_{:s}$ to denote the $s^{th}$ column of the matrix $\pi^m$.\\
The \emph{Method of Averaged Marginals}, depicted in Algorithm~\ref{alg_MAM}, leverages the transport plans 
from which one can extract marginals whose average is actually a barycenter approximation. 
The MAM algorithm asymptotically solves \eqref{WB} and produces both a barycenter and associate transport plans that are needed for the recursion problem \eqref{recursive_pb}. Note that Step 2 and 3 can be computed in parallel over the $M$ measures. We refer the interested reader to \cite{MAM} for more details.

\subsubsection{Computation of Transport Plans via Regularized Techniques} \label{pen_meth}

Utilizing efficient regularized methods enables us to quickly compute the transport plans required for the scenario tree reduction approach. For example, the \emph{Iterative Bregman Projection} (IBP) \cite{IBP} is a state-of-the-art technique, that applies regularization to the optimization problems presented in \eqref{WD} and \eqref{WB}.\\

\noindent Consider the entropic function:
\begin{equation}
    E(\pi) := \displaystyle \sum_{r,s}\pi_{r,s}(\log(\pi_{r,s})-1),
    \label{entropy}
\end{equation}
with the convention $0 \log(0) = 0$.
This is a strongly convex function which assures that $E(\pi)\ge0$ and $E(\pi)=0$ if and only if $\pi=0$. This function is employed to regularize the LP \eqref{WD}, leading to the following nonlinear optimization problem:
\begin{subequations}
\begin{align}
W_\lambda(\mu,\nu) &:= \displaystyle \min_{\pi \in U(\mu,\nu)} \langle D, \pi \rangle + \frac{1}{\lambda} E(\pi) \\
    &= \displaystyle \min_{\pi \in U(\mu,\nu)} \displaystyle \frac{1}{\lambda} \sum_{r,s} \left(\lambda D_{r,s}\pi_{r,s}+\pi_{r,s}\log(\pi_{r,s})-\pi_{r,s} \right)\\
    &= \displaystyle \min_{\pi \in U(\mu,\nu)} KL(\pi|K)
\end{align}\label{WDE}
\end{subequations}
Note that $KL$ is the Kullback-Leibler divergence between $\pi,K\in\R^{R\times S}$ and $K_{r,s}>0$ for all $(r,s)$: $ KL(\pi|K) := \displaystyle \sum_{r,s}\pi_{r,s} \left( \log \left(\frac{\pi_{r,s}}{K_{r,s}} \right) -1 \right)$ (the precise definition of matrix $K$ is given in the algorithm below).\\
By noticing that $p=\pi \mathds{1}_R$, problem \eqref{WB} with $W_\lambda$ introduced in \eqref{WDE}, instead of the classical Wasserstein distance, writes:
\begin{equation}\label{reg_sink_WD}
    \displaystyle \min_{\begin{tabular}{c}$\pi^m\in U(\mu,\nu^m)$,\\ $m=1,\dots,M$\end{tabular}} \sum_{m=1}^M \alpha_m KL(\pi^m|K^m)
\end{equation}
Since problem \eqref{reg_sink_WD} is 
strongly convex, it has a unique optimal solution. Following \cite{IBP}, the optimal coupling $\pi^m, m=1,\dots,M$ can be derived after iterative KL projections onto the right and left constraints, embodied by the sets $U(\mu,\nu^m), m=1,\dots,M$.\\

\begin{algorithm}
\caption{\sc IBP algorithm}
 
\begin{algorithmic}
\State {\bf Input}: Given $\alpha_m$ for $m=1,\dots,M$, $\lambda>0$, initialize $v^{0}$ and $u^{0}$  with an arbitrary positive vector, for example $\mathds{1}_S$
Initialize $p^0$, for example $\mathds{1}_R/R$
\State Set $D^{m} \gets \alpha_m \left( \delta_\iota (i,j) \right)_{(i,j)\in m+ \times n+}$ and set $q^m=\left(P(i|m)\right)_{i\in m+}$
\State Define $K^m=e^{-\lambda D^m}$ for all $m=1,\dots,M$
\While{not converged}
\Statex \Comment{Projections onto the constraints}
\For{m=1,\dots,M}
\State  $v^{m,k+1} = \frac{q^m}{(K^m)^T u^{m,k}} $
\State  $u^{m,k+1} = \frac{p^{k+1}}{K^m v^{m,k+1}}$
\EndFor
\Statex
\Statex \Comment{Approximation of the barycenter}
\State $p^{k+1} = \prod_{m=1}^M (K^m v^{m,k+1})^{\alpha_m}$
\EndWhile
\Statex
\State \Return $\pi^m = \diag(u^m)K^m\diag(v^m)$ for all $m=1,\dots,M$
\end{algorithmic} \label{alg_sink}
\end{algorithm}

Observe that all the algorithm's steps consist of matrix-vector multiplication, and are thus simple to execute. Note that $p^{k+1}$ is the current estimate of the barycenter in \Cref{alg_sink}. The algorithm's drawback is its accuracy, which strongly depends on $\lambda>0$. The greater is $\lambda$ the closer is the solution of \eqref{reg_sink_WD} to an exact solution of \eqref{WB}. However, if $\lambda$ is too large, the values of $K$ diverge, leading to computational issues such as double-precision overflow errors.

\subsection{Boosted Kovacevic and Pichler's Algorithm}
We rely on the previous subsections about Wasserstein barycenters computation methods to provide the following improved variant of the scenario tree reduction algorithm of  \cite{Kovacevic}. In Algorithm \ref{alg_reduc}, given a multistage scenario tree 
represented by $\mathbf{P}=(\Xi^{T+1}, \mathcal{F},P)$ 
a smaller scenario tree $\mathbf{P}'=(\Xi^{T+1}, \mathcal{F}',P')$ is constructed by updating the quantizer values $\{\xi(n)\in \Xi:n\in\mathcal{N}\}$ and probability $P'$, iteratively. The filtration $\mathcal{F}'$ is 
a data given to the algorithm. In other words, the number of scenarios and the structure of the reduced tree is an input data, and the algorithm seeks for the familly $\{\xi'(n)\in \Xi:n\in\mathcal{N}'\}$ and $P'$  that minimizes the nested distance between $\mathbf{P}$ and $\mathbf{P}'$.

\begin{algorithm}[h!]
\caption{\sc Scenario tree reduction via nested distance and Wasserstein barycenters}
  \label{alg_reduc}
\begin{algorithmic}[1]

\Statex \Comment{Step 0: input}
\State Let the original $T$-stage scenario tree $\mathbf{P}=(\Xi^{T+1}, \mathcal{F},P)$ and a smaller scenario tree
$\mathbf{P}^{'0}=(\Xi^{T+1}, \mathcal{F}',P^{'0})$ be given.
\State Compute a transport probabilities $\pi^0(i,j)$  between  scenarios $\left(\xi_{i}\right)_{i\in\mathcal{N}_T}$ and $\left(\xi^{'0}_{j}\right)_{j\in\mathcal{N}_T'}$\label{init_pi}

\State Set $k\gets 0$ and choose a tolerance $\tol>0$

\Statex  
\For{$k=1,2,\ldots$}
\Statex \Comment{\textbf{Step 1:} Improve the scenario values (quantizers)} \label{optim_quant}
\State 
    Set $\xi^{'k+1}(n_t)=\sum_{m\in \mathcal{N}_t}\frac{\pi^k(m,n_t)}{\sum_{i\in\mathcal{N}_t}\pi^k(i,n_t)}\xi_t(m)$ for all $n_t\in\mathcal{N}_t'$ for $t=1,\ldots,T$ 
 \Statex
\Statex \Comment{\textbf{Step 2:} Improve the probabilities}

\State Set $\delta_\iota^{k+1}(i, j)\gets\mathtt{d}(\xi_i, \xi_j)^{\iota}$ for all $i,j\in \mathcal{N}_T$
\For{$t=T-1,\dots, 0$} \Comment{Recursivity}
\For{ all $n \in \mathcal{N}'_t$} \Comment{Wasserstein barycenters} \label{parra}
\State Set $\alpha_m^n \gets \pi^k(m,n)$, $m\in \mathcal{N}_t$
\State Use IBP, or MAM to compute 
$\pi^{k+1}(\cdot,\cdot|\cdot,n)$  solving~\eqref{WB} 
\State Set $\delta_\iota^{k+1}(m,n)\gets \sum_{i\in m+, j\in n+}\pi^{k+1}(i,j|m,n)\delta_\iota^{k+1}(i,j)$, $m\in \mathcal{N}_t$
\EndFor

\EndFor
\Statex
\Statex \Comment{Build $\pi^{k+1}$ the unconditional transport plan matrix}
\State Set $\pi^{k+1}(0,0)\gets 1$
\For{t=1,\dots,T}
\State Compute $\pi^{k+1}(i,j)= \pi^{k+1}(i,j|m,n)\times \pi^{k+1}(m,n)$ for $i\in \mathcal{N}_t, j\in\mathcal{N}'_t$ and $m=i-$, $n=j-$
\EndFor

\Statex
\Statex \Comment{\textbf{Step 3:} Stopping test}
\If{$\delta_\iota^{k}(0,0) -  \delta_\iota^{k+1}(0,0) \leq \tol$}
\State Define $P'(n_T)= \sum_{m_T\in \mathcal{N}_T} \pi^{k+1}(m_T,n_T)$ for all $n_T \in \mathcal{N}'_T$ then $P'(n)= \sum_{j\in n+} P'(j)$ for all $n\in \mathcal{N}'_t, t\ne T$
\State Set ${\rm ND}_\iota(\mathbf{P}, \mathbf{P}')\gets \delta_\iota^{k+1}(0,0)$
\State Stop and return with the reduced tree $\mathbf{P}^{'}=(\Xi^{T+1},\mathcal{F}',P')$ and nested distance ${\rm ND}_\iota(\mathbf{P}, \mathbf{P}')$

\EndIf

\EndFor

  \end{algorithmic}
\end{algorithm}

Some comments on \Cref{alg_reduc} are in order.
\paragraph{Initialization} The probability $P^{'}$ and the scenario values $\{\xi'(n)\in \Xi:n\in\mathcal{N}'\}$ will be updated by the algorithm but the tree structure is fixed. 

We will provide some initialization methods in \Cref{impact_ini_sec}.
\\
Note that \cref{init_pi} is a straightforward operation since the initial transport plan only needs to respect the marginal constraints, therefore $\pi^0(i,j)=\frac{P(i|m)}{|n+|}$ for all $m,n,i\in m+, j\in n+$ is a sufficient initialization. 

\paragraph{Quantizer Optimizations} \Cref{alg_reduc} only considers the Euclidean case ($\iota=2$) in  \cref{optim_quant}, but if $\iota\ne 2$, the quantizer optimization step boils down to a gradient descent (see \cite{Kovacevic} for details) and the algorithm is still applicable although slower.

\paragraph{Hyperparameters} IBP relies on the use of a parameter $\lambda>0$ chosen by the user \cite{IBP,Cuturi}. Such parameter has an impact on the result's accuracy. The MAM algorithm also requires setting a parameter, but it only impacts the convergence speed and can be determined with a sensitivity analysis \cite{MAM}.

\paragraph{Parallelization} MAM is a parallelizable (and randomizable) algorithm, such a feature can also be leveraged in \Cref{alg_reduc}. \Cref{parra} can also be treated in a parallel manner.

\paragraph{Stopping criteria} The given stopping criteria  
is a heuristic: the algorithm terminates when the improvement of the nested distance between the two trees is below a certain level of tolerance $\tol$.

\if{
\paragraph{Storage complexity} The distance matrix $\Delta$ and the transport matrix $\pi$ whose size are equal to the number of nodes of the initial tree times the number of nodes in the approximate tree can be quite large to store but using adequate libraries, such as $numpy$ in $python$, allows for fast computation of the linear operations that are used in the method.
}\fi

\paragraph{Convergence} The algorithm leads to an improvement in each iteration. 
It should be kept in mind that the above algorithm is nothing but a heuristic, as it is  a block-coordinate scheme seeking to minimize the nested distance by alternating minimization over scenarios and then with respect to probabilities. 

\section{Applications} \label{Applications}
This section illustrates the performance of \Cref{alg_reduc} and its behaviour when using different solvers to compute Wasserstein barycenters (i.e. solutions of \eqref{recursive_pb} and \eqref{WB}). All the following applications employ \Cref{alg_reduc} in the Euclidean case, though similar conclusions can be extended to other values of $\iota$. If a standard LP solver is employed, then \Cref{alg_reduc} boils down to the setting of  \cite{Kovacevic}. 
We ensure that the considered solvers attain the same level of precision when tackling the Wasserstein barycenter problems at Step 2 of \Cref{alg_reduc}.  Numerical experiments were conducted using 20 cores (\textit{Intel(R) Xeon(R) Gold 5120 CPU}) and \textit{Python 3.9}, using a \textit{MPI} based parallelization when possible. The test problems and solvers’ codes are available for download at the link \url{https://github.com/dan-mim/Nested_tree_reduction}.

\subsection{Impact of the Tree Size}
In what follows, we consider scenarios trees composed by 4 to 8 stages and 5 to 6 children per nodes, offering numbers of scenarios and spanning from hundreds to ten thousands nodes (see \Cref{time_iteration} for details). The quantizers of both trees are set randomly within the range $[-10,10]$. 
In our preliminary tests, the reduced tree is always binary.
\\



We consider the following variants of the algorithm, by varying the solver used to compute the Wasserstein Barycenter in Step 2 of \Cref{alg_reduc}: 
\begin{itemize}
    \item \emph{LP}:  \Cref{alg_reduc} employing the LP solver HiGHS;
    \item \emph{MAM}:  \Cref{alg_reduc} employing  the \emph{Method of Averaged Marginals} \cite{MAM};
    \item \emph{IBP}:  \Cref{alg_reduc} employing IBP, inspired from the code of G. Peyré \cite{IBP_github}. This algorithm relies on the tune of an hyperparameter that has been preset to guarantee its best efficiency and convergence.
  
\end{itemize}
All variants use $\tol=0.1$ as a tolerance for the stopping test.
It has been empirically verified that results do not improve significantly if we decrease this tolerance further.

\begin{figure}[h!] 
    \centering
    \includegraphics[width=.8\textwidth]{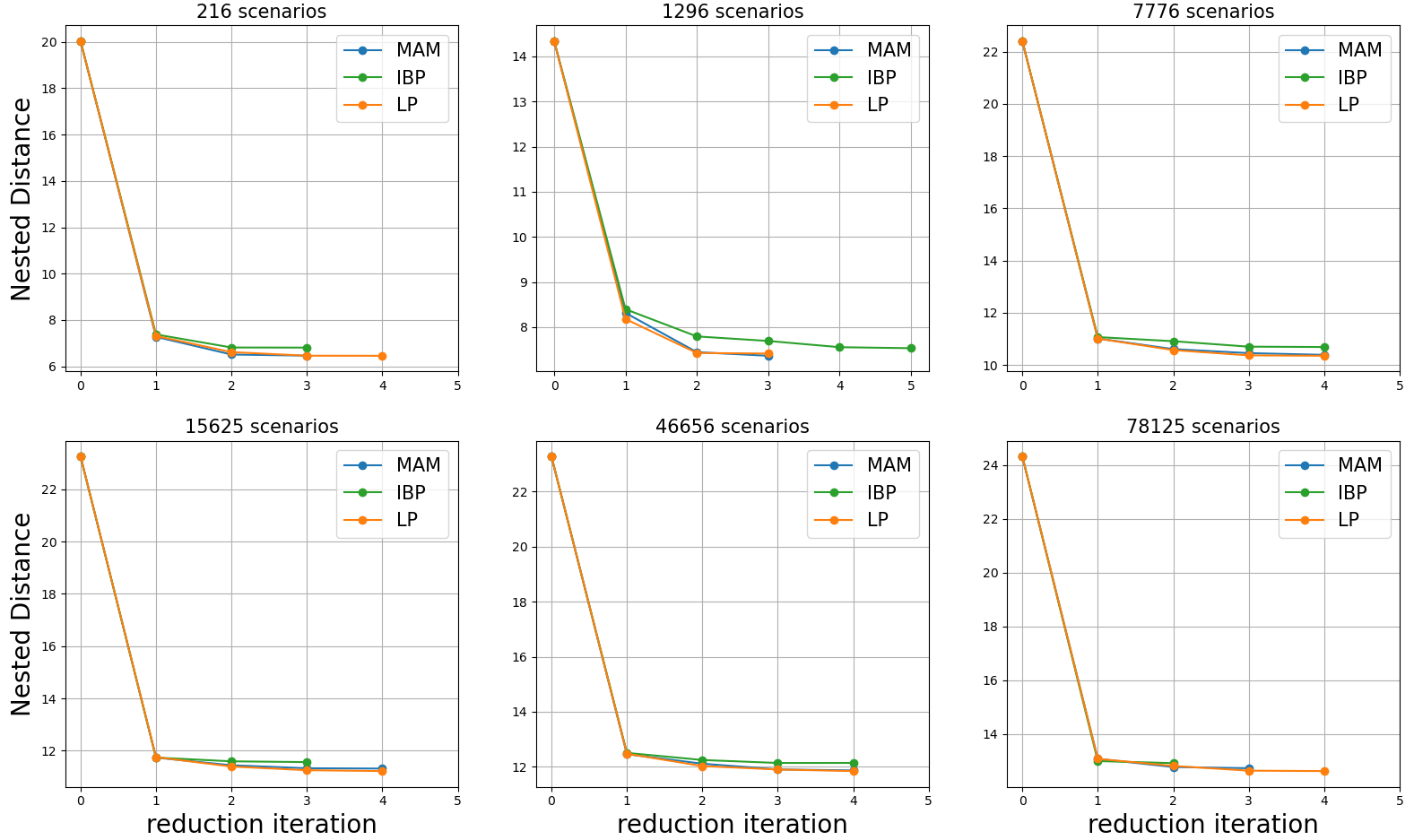}
    \caption{Evolution of the Nested Distance along the reduction iterations for different initial tree sizes. The final reduced tree has always two children per node.} \label{ND_iteration}
\end{figure}

\Cref{ND_iteration} shows that the exact ND between the original tree and the approximate one, iteratively decreases, no matter the variant of \Cref{alg_reduc}. 
All variants are initialized with the same tree, generated randomly (probabilities and scenarios). Note that the initial ND is always at least halved after the reduction. This emphasizes how important is the use of a reduction method. Within this scale, one can see that every variant converges to approximately the same precision, although not necessarily to the same reduced scenario tree (because different solvers compute different optimal transportation plans, impacting the construction of scenarios composing the reduced tree).

\begin{figure}[h!] 
    \centering
    \includegraphics[width=1\textwidth]{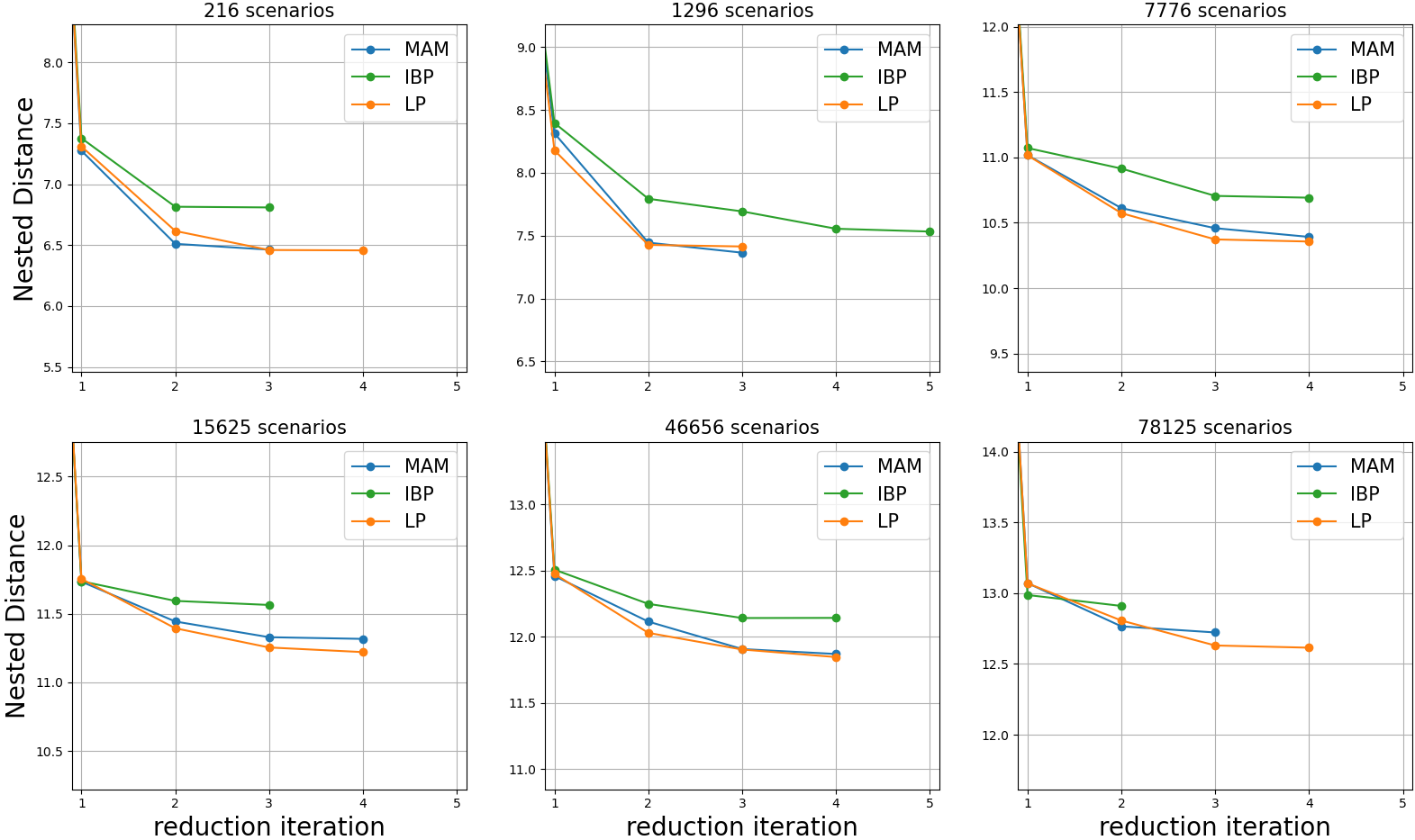}
    \caption{Evolution of the Nested Distance along the reduction iterations for different initial tree sizes with a zoom.} \label{ND_ZOOM}
\end{figure}

Even though the ND decreases with all variants, it seems faster (in terms of number of iterations) when using LP or MAM. \Cref{ND_ZOOM} shows that the IBP algorithm tends to reach a less precise plateau. 
This is due to the core of the IBP method, which is an inexact algorithm.
MAM being an exact algorithm for solving \eqref{WB}, it naturally follows the lead of LP. The slight differences between the variants LP and MAM can be explained by the fact that the general problem is non-convex and different optimal transportation plans computed by different solvers can lead to different optimization paths that result in different reduced scenario trees.
Therefore, depending on the employed solvers and the initialization, the approximated trees could be different while still having close ND with the original tree. 

\begin{table}[]
\centering
\begin{tabular}{|c|c|c|c|c|c|}
\hline
\multicolumn{1}{|l|}{{\color[HTML]{333333} \textbf{Scenarios}}} & {\color[HTML]{333333} \textbf{Nodes}} & {\color[HTML]{333333} \textbf{LP}} & {\color[HTML]{333333} \textbf{IBP}} & {\color[HTML]{333333} \textbf{MAM}} & {\color[HTML]{333333} \textbf{MAM 4 processors}} \\ \hline
\textbf{216} &  \textbf{259}                                & 0.17                               & 0.49                                & 2.21                                & 0.56                                             \\  \hline  
\textbf{1296} &  \textbf{1555}                                   & 1.54                               & 14.83                               & 18.23                               & 6.28                                             \\ \hline
\textbf{7776}  &  \textbf{9331}                                  & 74.25                              & 161.19                              & 344.83                              & 124.44                                           \\ \hline
\textbf{15625}  &  \textbf{19531}                                  & 487.58                             & 323.76                              & 816.46                              & 341.62                                           \\ \hline
\textbf{46656} &  \textbf{55987}                                   & 4905                               & 2136                                & 2541                                & 1256                                             \\ \hline
\textbf{78125} &  \textbf{97656}                                   & 13797                              & 4334                                & 3458                                & 1635                                             \\ \hline
\end{tabular}
\caption{Total time (in seconds) per method for the studied trees.} \label{time_iteration}
\end{table}

\Cref{time_iteration} shows that for small initial trees, up to 7776 scenarios, the LP variant is very efficient: it provides the lowest ND solution in the shortest time. But from 15625 scenarios and more, IBP is faster. As the number of scenarios increases, the relative performances of MAM get better and, in our case with more than 46656 scenarios, MAM is eventually twice faster than LP while reaching the same precision as depicted in \Cref{ND_iteration}. As shown in the last graph, with even more nodes and scenarios (78125 and 97656, respectively), MAM is the fastest variant.
Note that IBP is a very robust method: not only is the precision reached more than reasonable, according to \Cref{ND_iteration},  but the total time of execution is always in the ballpark of the fastest execution time of all algorithms.
Leveraging that MAM is parallelizable, we ran results using 4 processors. We witnessed that in this configuration, the variant MAM  is the most advantageous one when the initial tree has more than 20000 scenarios by far.

\subsection{Impact of the tree structure}\label{Structure_impact_section}
We observed when using the different approaches to tackle real-life examples that the \emph{structure} of the initial large tree has an impact on the convergence speed. Real-life trees do not span homogeneously from the root to the last stage, and when heterogeneity occurs, switching from one method to another can make the global reduction algorithm significantly faster.

\begin{figure}[h!] 
    \centering
    \includegraphics[width=.7\textwidth]{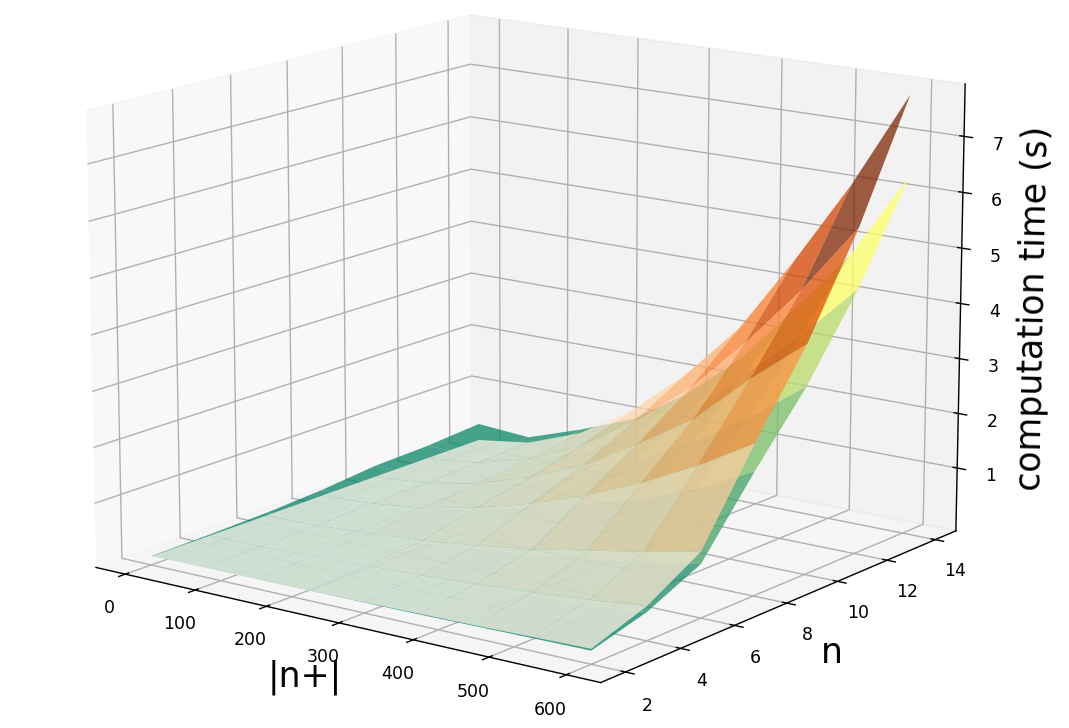}
    \caption{Influence of the tree structure on the computation time of a stage, depending on the method in use: MAM in \textcolor{green!50!black}{green} and LP in \textcolor{orange}{orange}.} \label{impact_structure}
\end{figure}


\Cref{impact_structure} illustrates the tree structure influence on the computation time of a stage, depending on the algorithm used. The original large tree is built as follows: from the root, two nodes are spanning (stage $t=2)$, then stage $t=3$ counts $n$ nodes. From this configuration, we made the number $|n+|$ of children from these $n$ nodes vary from 1 to 600 children (stage $t=T=4$). To put it differently, we built a tree with $n$ subtrees at \emph{stage $3$} having dimension $|n+|$. We reduce this tree into a binary one and evaluate \emph{stage $3$} reduction time\footnote{The computation time is evaluated as the mean of the five iterations along the reduction - at stage $3$}. We recall that the most expensive step of the algorithm consists in solving the greatest number of \eqref{WB} problems, thus stage 3. 

\begin{figure}[h!] 
    \centering
    \includegraphics[width=1\textwidth]{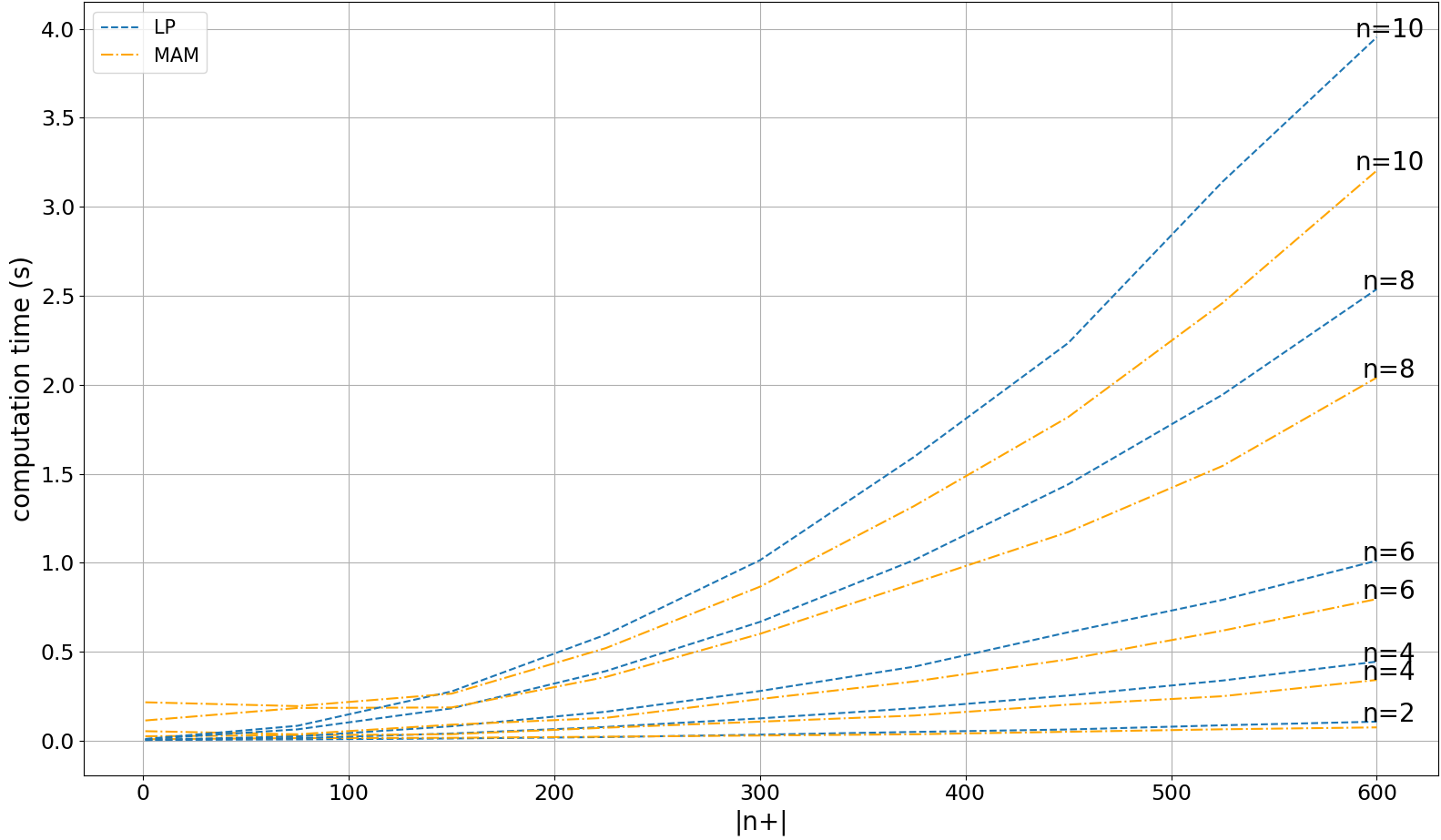}
    \caption{Influence of the tree structure on the computation time for small $n$.} \label{impact_structure_n_low}
\end{figure}

For any number $n$ of subtrees, the LP solver is always better when $|n+|$ is low, but when increasing $|n+|$  MAM solver ends up being faster because it treats better larger problems. This discrepancy gains momentum when the number of subtrees rises. It is observed that the treshold is reached even faster when $n$ is large. Therefore, at a stage with $n>10$ (and $|n+|>1$) where the reduction can take more than $4s$ we would advise to always use MAM. But, as illustrated in \Cref{impact_structure_n_low}, for smaller $n$ we would have a closer look, and use MAM only if $|n+|>150$, otherwise keep the LP method. \Cref{impact_structure} and \Cref{impact_structure_n_low} show that choosing the adequate method to tackle the reduction can speed up the computation time from $20\%$ to $35\%$. In practice, it can be observed that this repartition ends up using half MAM (mostly for the deep stages where $t\gg 0$) and half LP (mostly for the stages close to the root and the heterogeneity in the structure) to reduce a real-life initial tree.

The IBP method is challenging to use in practice within the tree reduction algorithm because of the complexity involved in tuning the hyperparameter $\lambda$, which ideally needs to be carefully chosen for each barycenter computation to achieve acceptable precision. Fine-tuning is necessary to control its accuracy. If a broadly set hyperparameter $\lambda$ is used, the algorithm may either encounter double-precision overflow errors at certain stages or compute a solution that significantly deviates from the exact optimization, without any guarantee or control over the resulting precision.
\begin{figure}[h]
    \centering
    \begin{subfigure}[t]{0.45\textwidth}
        \centering
        \includegraphics[width=\linewidth]{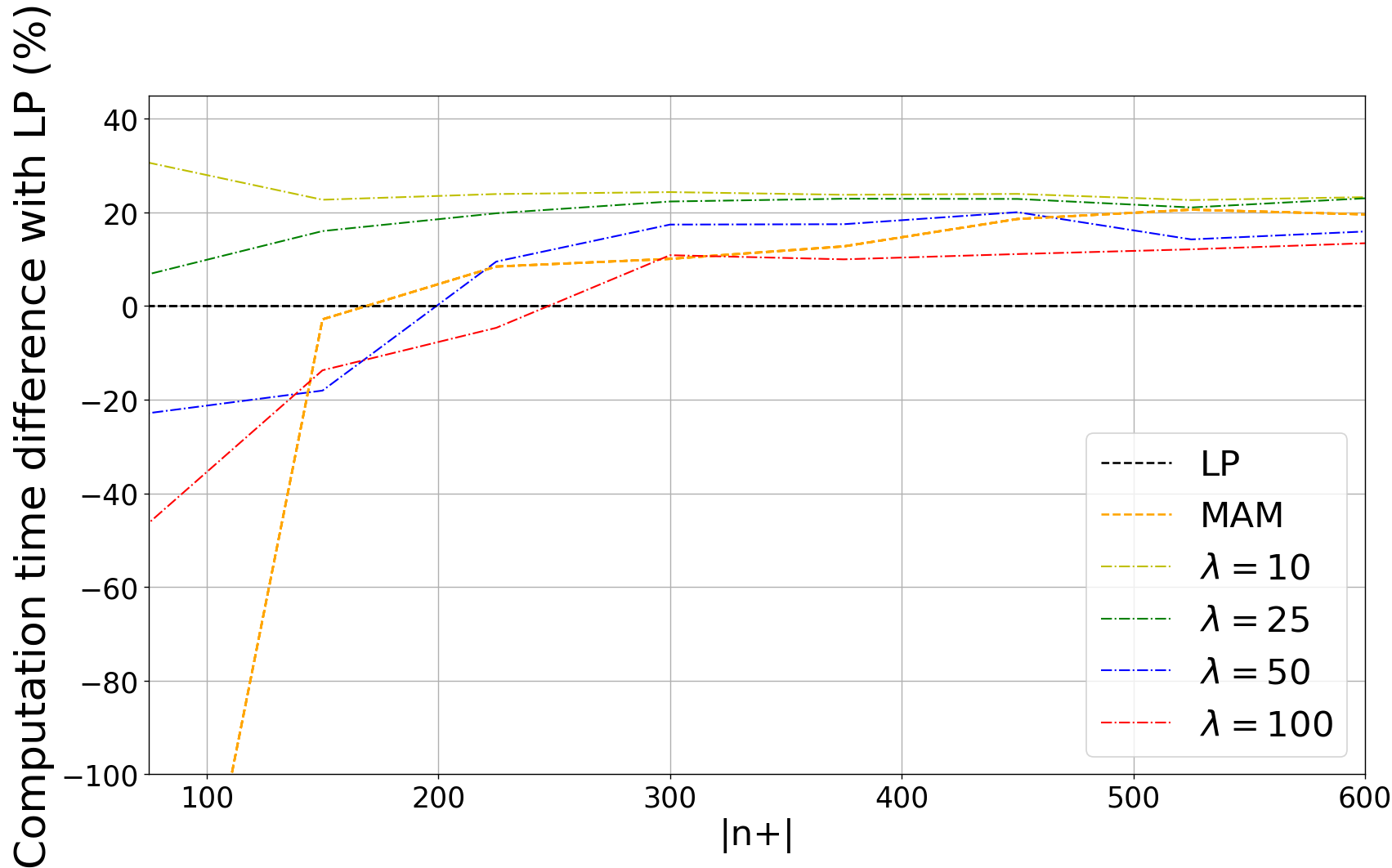}
        \caption{$n=8$}
    \end{subfigure}
    ~
    \begin{subfigure}[t]{0.45\textwidth}
        \centering
        \includegraphics[width=\linewidth]{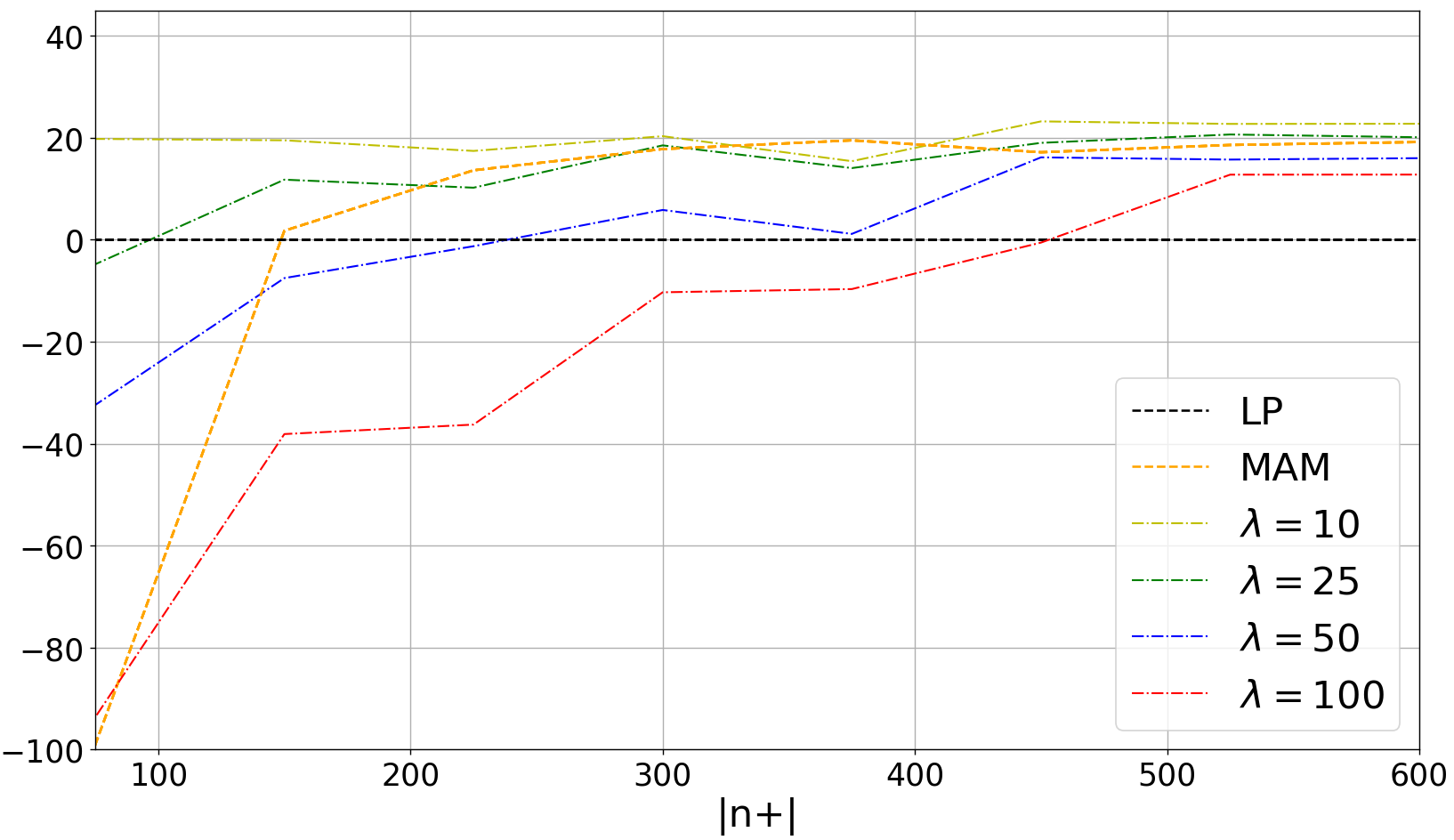}
        \caption{$n=12$}
    \end{subfigure}
    \caption{Speed comparison with IBP for different $\lambda$: A positive time difference means the method is faster than LP. Each curve is obtained by averaging the ND accuracy over $n\in\{2,4,6,8,10,12,14,16\}$. 
    }
    \label{IBP_lambda}
\end{figure}

\begin{figure}[h!] 
    \centering
    \includegraphics[width=1\textwidth]{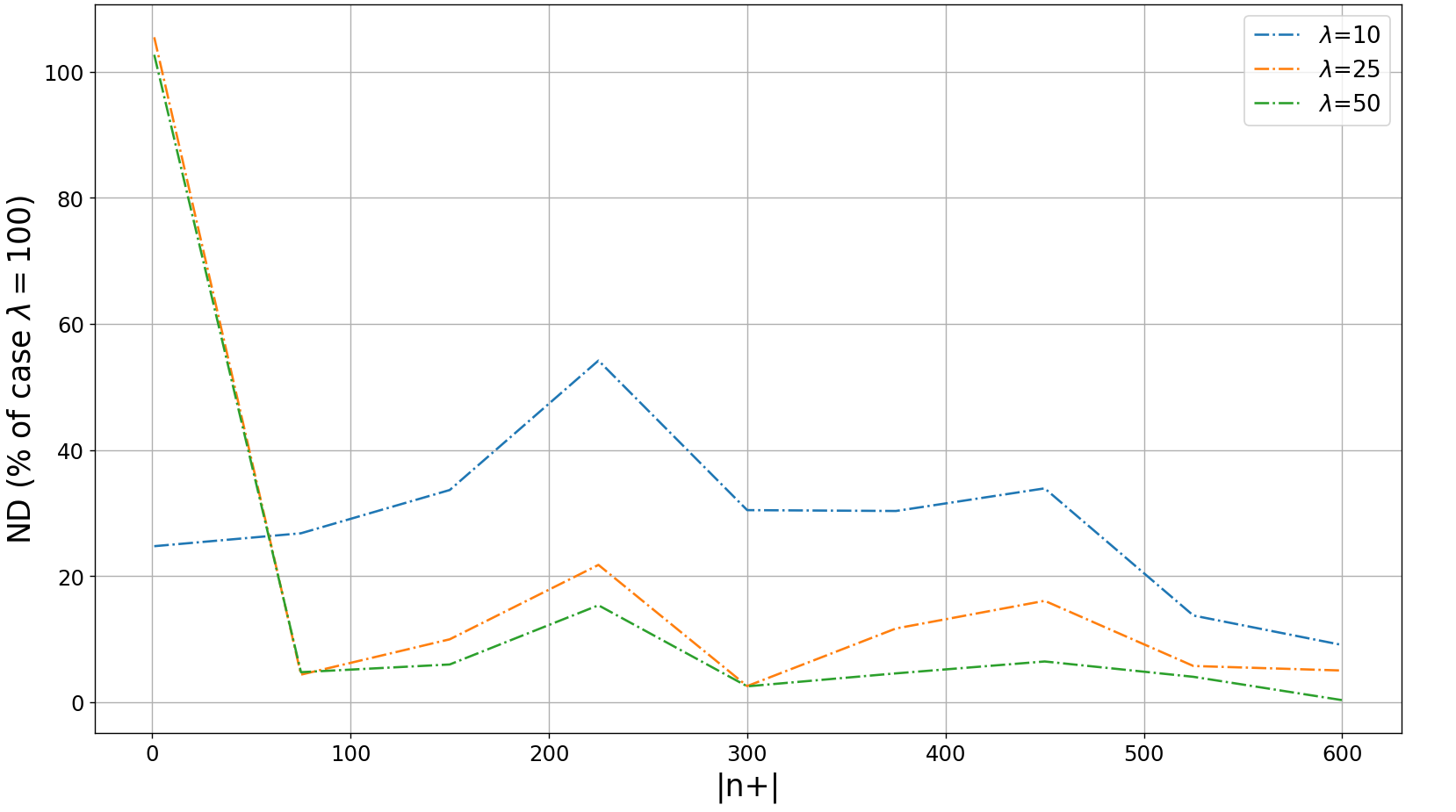}
    \caption{Average influence of $\lambda$ in the precision. Each curve is obtained by averaging the ND accuracy over $n\in\{2,4,6,8,10,12,14,16\}$.} \label{IBP_acc}
\end{figure}

\Cref{IBP_lambda} presents a study on the influence of a broad $\lambda$ used for IBP in the tree reduction algorithm for the same datasets as earlier. It shows that the use of IBP enables the reduction tree algorithm to be fast for both small and large $n$ and $n+$, and stresses that the smaller the $\lambda$ the faster the reduction. But \Cref{IBP_acc} underlines that this speed comes with a cost since the smaller the $\lambda$ and the less accurate the computation of the barycenter - the greater the ND. \Cref{IBP_acc} has been obtained by averaging the ND accuracy over $n$ \footnote{$n\in\{2,4,6,8,10,12,14,16\}$.}, where the case $\lambda=100$ is taken as reference.

\subsection{Impact of the Initialization}\label{impact_ini_sec}
In this section, we incorporate the reduction tree algorithm (utilizing either MAM or LP, depending on the results from \Cref{Structure_impact_section}) into a stochastic optimization pipeline. Data were collected from an industrial site in Solaize, France, where battery production and local consumption have been monitored over several years. Using this data, we generated scenarios for battery consumption and production throughout a single day, divided into 48 time steps, employing the methodology outlined in \cite{amabile2021optimizing}. This approach allows us to create 100 $2D$ scenarios, each consisting of 48 stages with scenario values spanning from 0 to 25kW for the consumption and 0 to 40kW for the production.
These scenarios can be modeled as a large tree. We reduce this tree using various initial filtrations to construct a smaller, approximate tree.
To built these filtrations we leverage two scenario selection algorithms:
\begin{itemize}
    \item \emph{Kmeans method}, starting from 100 scenarios it creates 25 clusters using the Euclidean norm, and then computes the 25 corresponding barycenters;
    \item \emph{Fast Forward Selection (FFS) method}, introduced by Heitsch and Römisch in \cite{heitsch_romisch}. The method iteratively selects scenarios that minimize the Wasserstein distance to the remaining scenarios.
    At each step, the scenario that best approximates the distribution is added to the reduced set until the desired number of 25 scenarios is reached, ensuring an efficient yet effective reduction.
\end{itemize}

From these reduced number of scenarios we generate the associate trees. They will initialize \Cref{alg_reduc}. We make the experience multiple times by generating several sets of scenarios.

\begin{figure}
    \centering
    \captionsetup{type=table} 
    \includegraphics[width=.8\linewidth]{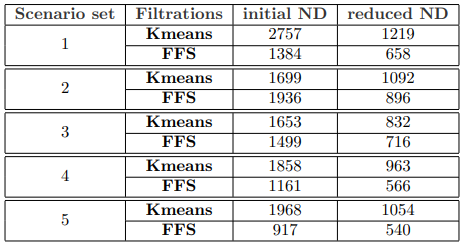}
    \caption{Comparison of the ND to the original tree before and after tree reduction using different initialization techniques.}
    \label{tab0}
\end{figure}

\Cref{tab0} compares the ND to the original large tree, both before and after tree reduction, highlighting the importance of the initial filtration.
Note that the ND is consistently  reduced after the tree reduction algorithm, as expected. 
When FFS is used to generate the initial filtration, the reduced approximated tree remains consistently closer to the original tree in terms of nested distance, even if the ND between the original tree and the FFS filtration is larger than that between the original tree and the K-means filtration at initialization. This is because FFS is a specialized algorithm designed to minimize the Wasserstein distance between the selected scenarios and the original ones, thereby preserving the maximum amount of information from the initial processes.